\documentclass[11pt]{amsart}

\usepackage{amsmath}
\usepackage{amscd}
\usepackage{amssymb}
\usepackage{boxedminipage}
\usepackage{enumerate}
\usepackage{footmisc}
\usepackage[all]{xy}
\usepackage{graphics}
\usepackage{epic}

\begin{document}

\theoremstyle{plain} \newtheorem{theorem}{Theorem}[section]
\theoremstyle{plain} \newtheorem{lemma}[theorem]{Lemma}
\theoremstyle{plain} \newtheorem{proposition}[theorem]{Proposition}
\newtheorem{axioms}[theorem]{Axioms}
\newtheorem{corollary}[theorem]{Corollary}
\newtheorem{problem}{Problem}
\newtheorem{subproblem}[problem]{Subproblem}
\newtheorem{conjecture}[theorem]{Conjecture}
\newtheorem{conjecture*}[]{Conjecture}
\newtheorem{matheorem}[theorem]{Main Theorem}

\newcommand{\nr}{\refstepcounter{theorem}  
                   \noindent {\thetheorem .}}
\newcommand{\defi}{\medskip \noindent {\it Definition \nr} }
\newcommand{\defifin}{\medskip}
\newcommand{\eks}{\medskip \noindent {\it Example \nr} }
\newcommand{\eksfin}{\medskip}
\newcommand{\rem}{\medskip \noindent {\it Remark \nr} }
\newcommand{\remfin}{\medskip}
\newcommand{\obs}{\medskip \noindent {\it Observation \nr} }
\newcommand{\obsfin}{\medskip}
\newcommand{\note}{\medskip \noindent {\it Notation \nr} }
\newcommand{\notefin}{\medskip}
\renewcommand{\thefootnote}{\fnsymbol{footnote}}

\newcommand{\llabel}{\addtocounter{theorem}{-1}
\refstepcounter{theorem} \label}

\newcommand{\psp}[1]{{{\bf P}^{#1}}}
\newcommand{\psr}[1]{{\bf P}(#1)}
\newcommand{\op}{{\mathcal O}}
\newcommand{\opw}{\op_{\psr{W}}}
\newcommand{\go}{\op}

\newcommand{\ini}[1]{\text{in}(#1)}
\newcommand{\gin}[1]{\text{gin}(#1)}
\newcommand{\kr}{{\Bbbk}}
\newcommand{\pd}{\partial}
\renewcommand{\tt}{{\bf t}}


\newcommand{\coh}{{{\text{{\rm coh}}}}}


\newcommand{\modv}[1]{{#1}\text{-{mod}}}
\newcommand{\modstab}[1]{{#1}-\underline{\text{mod}}}

\newcommand{\sut}{{}^{\tau}}
\newcommand{\sumit}{{}^{-\tau}}
\newcommand{\til}{\thicksim}

\newcommand{\totp}{\text{Tot}^{\prod}}
\newcommand{\dsum}{\bigoplus}
\newcommand{\dprod}{\prod}
\newcommand{\lsum}{\oplus}
\newcommand{\lprod}{\Pi}

\newcommand{\La}{{\Lambda}}

\newcommand{\sirstj}{\circledast}

\newcommand{\she}{\EuScript{S}\text{h}}
\newcommand{\cm}{\EuScript{CM}}
\newcommand{\cmd}{\EuScript{CM}^\dagger}
\newcommand{\cmri}{\EuScript{CM}^\circ}
\newcommand{\cler}{\EuScript{CL}}
\newcommand{\clerd}{\EuScript{CL}^\dagger}
\newcommand{\clerri}{\EuScript{CL}^\circ}
\newcommand{\gor}{\EuScript{G}}
\newcommand{\gF}{\mathcal{F}}
\newcommand{\gM}{\mathcal{M}}
\newcommand{\gE}{\mathcal{E}}
\newcommand{\gD}{\mathcal{D}}
\newcommand{\gI}{\mathcal{I}}
\newcommand{\gP}{\mathcal{P}}
\newcommand{\gK}{\mathcal{K}}
\newcommand{\gL}{\mathcal{L}}
\newcommand{\gS}{\mathcal{S}}
\newcommand{\gC}{\mathcal{C}}
\newcommand{\gO}{\mathcal{O}}
\newcommand{\gJ}{\mathcal{J}}

\newcommand{\dlim} {\varinjlim}
\newcommand{\ilim} {\varprojlim}



\newcommand{\Kom}{\text{Kom}}


\newcommand{\EH}{{\mathbf H}}
\newcommand{\res}{\text{res}}
\newcommand{\Hom}{\text{Hom}}
\newcommand{\inhom}{{\underline{\text{Hom}}}}
\newcommand{\Ext}{\text{Ext}}
\newcommand{\Tor}{\text{Tor}}
\newcommand{\ghom}{\mathcal{H}om}
\newcommand{\gext}{\mathcal{E}xt}
\newcommand{\id}{\text{{id}}}
\newcommand{\im}{\text{im}\,}
\newcommand{\codim} {\text{codim}\,}
\newcommand{\resol}{\text{resol}\,}
\newcommand{\rank}{\text{rank}\,}
\newcommand{\lpd}{\text{lpd}\,}
\newcommand{\coker}{\text{coker}\,}
\newcommand{\supp}{\text{supp}\,}


\newcommand{\sus}{\subseteq}
\newcommand{\sups}{\supseteq}
\newcommand{\pil}{\rightarrow}
\newcommand{\vpil}{\leftarrow}
\newcommand{\rpil}{\leftarrow}
\newcommand{\lpil}{\longrightarrow}
\newcommand{\inpil}{\hookrightarrow}
\newcommand{\pils}{\twoheadrightarrow}
\newcommand{\projpil}{\dashrightarrow}
\newcommand{\dotpil}{\dashrightarrow}
\newcommand{\adj}[2]{\overset{#1}{\underset{#2}{\rightleftarrows}}}
\newcommand{\mto}[1]{\stackrel{#1}\longrightarrow}
\newcommand{\vmto}[1]{\stackrel{#1}\longleftarrow}
\newcommand{\eqv}{\Leftrightarrow}
\newcommand{\impl}{\Rightarrow}

\newcommand{\iso}{\cong}
\newcommand{\te}{\otimes}
\newcommand{\into}[1]{\hookrightarrow{#1}}
\newcommand{\ekv}{\Leftrightarrow}
\newcommand{\equi}{\simeq}
\newcommand{\isopil}{\overset{\cong}{\lpil}}
\newcommand{\equipil}{\overset{\equi}{\lpil}}
\newcommand{\ispil}{\isopil}
\newcommand{\vvi}{\langle}
\newcommand{\hvi}{\rangle}


\newcommand{\xd}{\check{x}}
\newcommand{\ortog}{\bot}
\newcommand{\tL}{\tilde{L}}
\newcommand{\tM}{\tilde{M}}
\newcommand{\tH}{\tilde{H}}
\newcommand{\tvH}{\widetilde{H}}
\newcommand{\tvh}{\widetilde{h}}
\newcommand{\tV}{\tilde{V}}
\newcommand{\tS}{\tilde{S}}
\newcommand{\tT}{\tilde{T}}
\newcommand{\tR}{\tilde{R}}
\newcommand{\tf}{\tilde{f}}
\newcommand{\ts}{\tilde{s}}
\newcommand{\tp}{\tilde{p}}
\newcommand{\tr}{\tilde{r}}
\newcommand{\tfst}{\tilde{f}_*}
\newcommand{\empt}{\emptyset}
\newcommand{\bfa}{{\bf a}}
\newcommand{\la}{\lambda}

\newcommand{\ome}{\omega_E}

\newcommand{\bevis}{{\bf Proof. }}
\newcommand{\demofin}{\qed \vskip 3.5mm}
\newcommand{\nyp}[1]{\noindent {\bf (#1)}}
\newcommand{\demo}{{\it Proof. }}
\newcommand{\demodone}{\demofin}
\newcommand{\parg}{{\vskip 2mm \addtocounter{theorem}{1}  
                   \noindent {\bf \thetheorem .} \hskip 1.5mm }}


\newcommand{\dl}{\Delta}
\newcommand{\cdel}{{C\Delta}}
\newcommand{\cdelp}{{C\Delta^{\prime}}}
\newcommand{\dlst}{\Delta^*}
\newcommand{\Sdl}{{\mathcal S}_{\dl}}
\newcommand{\lk}{\text{lk}}
\newcommand{\lkd}{\lk_\Delta}
\newcommand{\lkp}[2]{\lk_{#1} {#2}}
\newcommand{\del}{\Delta}
\newcommand{\delr}{\Delta_{-R}}
\newcommand{\dd}{{\dim \del}}

\renewcommand{\aa}{{\bf a}}
\newcommand{\bb}{{\bf b}}
\newcommand{\cc}{{\bf c}}

\newcommand{\pnm}{{\bf P}^{n-1}}
\newcommand{\opnm}{{\go_{\pnm}}}
\newcommand{\ompnm}{\omega_{\pnm}}

\newcommand{\pn}{{\bf P}^n}
\newcommand{\hele}{{\bf Z}}

\newcommand{\dt}{{\displaystyle \cdot}}
\newcommand{\st}{\hskip 0.5mm {}^{\rule{0.4pt}{1.5mm}}}              
\newcommand{\disk}{\scriptscriptstyle{\bullet}}

\def\CC{{\mathbb C}}
\def\GG{{\mathbb G}}
\def\ZZ{{\mathbb Z}}
\def\NN{{\mathbb N}}
\def\OO{{\mathbb O}}
\def\QQ{{\mathbb Q}}
\def\VV{{\mathbb V}}
\def\PP{{\mathbb P}}
\def\EE{{\mathbb E}}
\def\FF{{\mathbb F}}
\def\AA{{\mathbb A}}

\newcommand{\MaFigNog}{\hskip 1cm
\begin{picture}(200,105)
\put(40,4){$x_1^{a+b+1}$} \put(56,18){x}
\dottedline{6}(70,20)(94,20)
\put(104,18){x}
\put(116,18){x}
\put(131,21){\circle{6}}
\dottedline{6}(140,20)(150,20)
\put(156,21){\circle{6}}
\put(164,4){$x_2^{a+b+1}$}

\put(64,30){x}
\dottedline{6}(76,32)(92,32)
\put(96,30){x}
\put(108,30){x}
\put(123,33){\circle{6}}
\dottedline{6}(130,32)(140,32)
\put(148,33){\circle{6}}

\dottedline{6}(80,50)(90,65)
\dottedline{6}(130,50)(120,65)

\put(105,78){\circle{6}}
\put(101,88){$x_3^{a+b+1}$}

\end{picture}
}

\newcommand{\MaFigGape}{\hskip 2cm
\begin{picture}(200,100)
\put(40,12){$x_1^{4}$} \put(56,18){x}
\put(70,18){x}
\put(84,18){x}
\put(98,18){x}
\put(114,21){\circle{6}}
\put(122,14){$x_2^4$}

\put(63,30){x}
\put(77,30){x}
\put(93,33){\circle{6}}
\put(107,33){\circle{6}}

\put(70,42){x}
\put(86,45){\circle{6}}
\put(100,45){\circle{6}}

\put(79,57){\circle{6}}
\put(93,57){\circle{6}}

\put(86,69){\circle{6}}

\put(82,79){$x_3^4$}
\end{picture}
}

\newcommand{\MaFigEkstra}{\hskip 2cm
\begin{picture}(200,80)
\put(47,4){$x_1^{3}$} 
\put(115,6){$x_2^3$}

\put(63,10){x}
\put(77,10){x}
\put(91,10){x}
\put(107,13){\circle{6}}

\put(70,22){x}
\put(86,25){\circle{6}}
\put(100,25){\circle{6}}

\put(79,37){\circle{6}}
\put(93,37){\circle{6}}

\put(86,49){\circle{6}}

\put(82,59){$x_3^3$}
\end{picture}
}

\title [Geometric properties derived from generic initial spaces]
{Geometric properties derived from generic  initial spaces}
\author {Gunnar Fl{\o}ystad \and Mike Stillman}
\footnotetext{The second author partially supported by NSF Grant DMS-0311806.
 }
\address{ Matematisk Institutt \\ Johs. Brunsgt. 12 \\ 5008 Bergen \\ Norway }
\address{Dep. of Mathematics\\
          Cornell University\\
          Ithaca, NY 14853 \\
          USA}

\email{ gunnar@mi.uib.no \and mike@math.cornell.edu}

\begin{abstract}
For a vector space $V$ of homogeneous forms of the same degree
in a polynomial ring, we investigate what can be said about the
generic initial ideal 
of the ideal generated by $V$, from the form of the generic
initial space $\gin{V}$ for the revlex order. Our main result is 
a considerable generalisation of a previous result by the first author.
\end{abstract}

\maketitle

\section{Introduction}

Generic initial ideals for the revlex order is a basic invariant of 
a homogeneous ideal $I$ in a polynomial ring. The computational complexity of 
the ideal is reflected in this generic initial ideal, as well as algebraic
and geometric properties of the original ideal. 
For instance, this generic initial ideal has the same regularity as the original
ideal \cite{BaSt}. Also the original ideal is Cohen-Macaulay iff this generic
initial ideal is Cohen-Macaulay. In \cite{Gr}, M.Green shows that if $\gin{I}$
has no generator in degree $e$ and $I_{< e}$ is the subideal of $I$ generated
by elements of degree less than $e$, then the generic initial ideal of this 
subideal
is in fact equal to $\gin{I}_{<e}$, forcing strong conditions on the geometry
of the scheme defined by $I$. With such basic properties of the ideal reflected
in the generic initial ideal it is of considerable interest to understand more
of the what properties of the original ideal are reflected in the generic 
initial ideal for the revlex order.

We consider subspaces $V$ of homogeneous forms in 
a polynomial ring $\kr[x_1, \ldots, x_n]$ and investigate what geometric consequences one
may derive about the vanishing locus of $V$ from its generic initial space $\gin{V}$, for 
the revlex order. In particular our results will apply to the ideal generated
by the forms in $V$.

A result by the first author \cite{Fl}, says that if $V$ is a linear subspace of forms
of degree $a+b$ and its generic initial subspace for the revlex order has as basis of 
monomials, the strongly stable set of monomials of degree $a+b$ generated by 
$x_1^a x_m^b$, where $m \geq 3$, then 
the polynomials of $V$ have a common factor of degree $a$.

We give a generalisation of this, Theorem \ref{MaTheMa}, putting a considerably
weaker condition on $\gin{V}$ and from this deriving conclusions about the growth
of the Hilbert function of the ideal generated by $V$. 
For instance, the case when $m=n$ above is generalised as follows.  
The property we impose is that every monomial of 
$\gin{V}$ is a multiple of a monomial in $\gin{V} : x_n$. This enables us to 
conclude that the image in $\kr [x_1, \ldots, x_{n-1}]$ of the ideal generated by $V$,
and the ideal generated by $\gin{V}$, have the same Hilbert function.

In the end we consider the example where $\gin{V}$ has as basis the  monomials 
$x_1^3, x_1^2x_2, x_1x_2^2, x_1^2 x_3$ of $\kr [x_1, x_2, x_3]$. We give examples showing 
that in this
case the vanishing locus of $V$ may consist of six points, three points, or be empty.

\section{The results}

We first recall basic notions and properties of the revlex order
and initial ideals. Then we formulate the main results, and in 
the end we give some examples and open ends.

\subsection{The revlex order}
We consider the revlex order on monomials in the polynomial ring
$\kr [x_1, \ldots, x_n]$ over a field $\kr$ of characteristic zero. I.e. 
\[ x_1^{i_1} \cdots x_n^{i_n} > x_1^{j_1} \cdots x_n^{j_n} \]
if either the total degree of the former is less than the total degree
of the latter, or, if their total degrees are equal, for the largest
index $r$ such that $i_r$ and $j_r$ differ, we have $i_r < j_r$.

For an ideal $I$ in the polynomial ring, we get the initial ideal 
$\ini{I}$ generated by the largest monomials in each polynomial of $I$.
Also for any vector subspace $V$ of the polynomial ring, we get a 
subspace $\ini{V}$, generated as a vector space by the largest monomials
in the polynomials of $V$.

For a vector subspace $V$ of $\kr [x_1, \ldots, x_n]$, the image of $V$
by the natural map
\[ \kr [x_1, \ldots, x_n] \pil \kr [x_1, \ldots, x_{n-r}] \]
is denoted by $V_{|x_n,\ldots, x_{n-r+1}}$, and called the restriction
of $V$ to $\kr [x_1, \ldots, x_{n-r}]$. If $I$ is an ideal, its
restriction is again an ideal.

\medskip
Two standard facts about the revlex order are the following.
\begin{itemize}
\item[i.] $\ini{V_{|x_n}} = \ini{V}_{|x_n}$,

\item[ii.] $\ini{V:x_n} = \ini{V} : x_n$.
\end{itemize}
\medskip

\subsection{Generic initial spaces} If a sequence of polynomials are surrounded by
brackets $\langle, \rangle$ it will denote the vector space generated by
these polynomials.
If $g$ in $GL(\langle x_1, \ldots, x_n \rangle)$ is a change of coordinates,
there is an open subset of such $g$'s such that $\ini{g.V}$ is constant.
This initial space is called the generic initial space of $V$ and denoted
$\gin{V}$. It is a standard fact, see \cite{Ei}, or originally, \cite{Ga},
that $\gin{V}$ is strongly stable, i.e. if $m$ is a monomial in $\gin{V}$, 
$x_j$ divides $m$ and $i < j$, then $x_i m /x_j$ is also in $\gin{V}$.

Let us recall the following result by the first author \cite{Fl}.

\begin{theorem} \label{MaTheFl}
Let $V \sus S_{b+a}$ be a linear subspace such that
\[ \gin{V} = \langle x_1, \ldots, x_{m} \rangle ^b \cdot x_1^a \sus S_{b+a}, \]
where $m \geq 3$. Then there exists a polynomial $p$ in $S$ of degree $a$ such that $p$ 
is a common factor of the elements of $V$.
\end{theorem}

The theorem we prove here is a considerable generalisation of this. First
let us give an illustrative example of this generalisation.

\eks Suppose $V$ is a subspace of $\kr[x_1, x_2, x_3]_{a+b+1}$ such that $x_1^ax_2^{b+1}$ is 
not in $\gin{V}$, but $x_1^{a+1}x_2^{b-1}x_3$ is in $\gin{V}$ (and so also $x_1^{a+1}x_2^b$).
This means that $\gin{V}$ has the form given in the following figure, where an x denotes
that the monomial is in $\gin{V}$, and a circle that it is not.

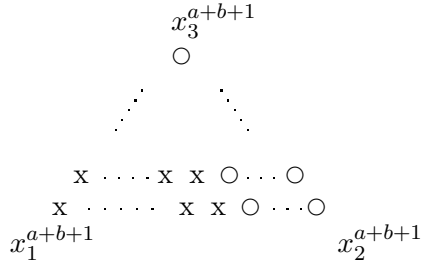
\MaFigNog

\noindent By the result of the first author, Theorem \ref{MaTheFl}
if $\gin{V}$ is $x_1^{a+1}\cdot \langle x_1,x_2,x_3 \rangle^b$ then $V$ has a common factor $p$ of degree
$a+1$. But this assumption on $V$ is very strong. The assumption above is much weaker
and we shall show that it still allows us to conclude that $V$ has a common factor.
\eksfin

In addition to the theorem above, the following, \cite[Prop.28]{Gr}, 
due to M.Green is a main inspiration for what we prove.

\begin{theorem} \label{MaTheGr}
Let $I$ be an ideal in $\kr[x_1, \ldots, x_n]$ 
generated in degrees $\leq d$. 
Suppose $\ini{I}$ is strongly stable and has no generator in degree $d$.
Then $\ini{I}$ is generated in degree $<d$.
\end{theorem}  

Green shows this for the generic initial ideal $\gin{I}$. However
one readily sees that his argument holds under the conditions that
$\ini{I}$ is strongly stable.

\subsection{Main result.}
Now if $T$ is a subspace of $\kr [x_1, \ldots, x_n]$ spanned by monomials of degree $d$, 
and $T$ is strongly stable, then
\[ (T : x_n^r)_{|x_n} \cdot \langle x_1, \ldots, x_{n-1} \rangle \sus 
(T:x_n^{r-1})_{|x_n}.\]
Thus, if $I_T$ denotes the ideal generated by $T$, then 
\[ \oplus_{r=1}^d (T:x_n^r)_{|x_n} \oplus (I_T)_{|x_n} \]
becomes an ideal in $\kr [x_1, \ldots, x_{n-1}]$ which we denote by 
$\gJ(T)_{x_n}$.

\begin{matheorem} \label{MaTheMa} Let $V$ be a subspace of $\kr [x_1, \ldots, x_n]$ 
consisting of polynomials of degree $d$. Suppose $V$ is in general coordinates and 
let $I_V$ be the ideal generated by $V$.

a. If $(\ini{V}:x_n)\cdot \langle x_1, \ldots, x_n \rangle = \ini{V}$, 
i.e. the ideal 
$\gJ(\ini{V})_{x_n}$ has no generator in degree $d$, 
the initial ideal $\ini{(I_V)_{|x_n}}$ 
will be equal to this ideal in degrees $\geq d$

b. More generally, if the ideal $\gJ(\ini{V}_{|x_{n}, \ldots, x_{n-r+1}})_{x_{n-r}}$
has no generator in degree
$d$, the initial ideal $\ini{(I_V)_{|x_n, \ldots, x_{n-r}}}$ will be equal to this
ideal in degrees $\geq d$.

\end{matheorem}

\begin{proof}
Let $\kr(\tt)$ be the function field of $\kr [t_1, \ldots, t_n]$, and let $h = \sum_{i=1}^n t_i x_i$
in $\kr(\tt)[x_1, \ldots, x_n]$. Also let $V_{\kr(\tt)} = V \te_\kr \kr(\tt)$.

\medskip
\noindent{\it Claim.} The direct sum of spaces 
\begin{equation} \oplus_{r=1}^d (V_{\kr(\tt)} : h^r)_{|h} \oplus (I_{V_{\kr(\tt)}})_{|h} 
 \sus \kr(\tt)[x_1, \ldots, x_n]/(h). \label{MaLaKtx}
\end{equation}
is an ideal in the ring on the right side.

\medskip
This follows because
\[ \langle x_1, \ldots, x_n \rangle \cdot (V_{\kr(\tt)} : h^r)_{|h} \sus
(V_{\kr(\tt)} : h^{r-1})_{|h} . \]
To see this let $p$ be in $(V_{\kr(\tt)} : h^r)$. Then $ph^r$ is in $V_{\kr(\tt)}$. Now if a vector varies
in a vector space, here $V$, the derivative of the vector will also be in the vector space. Hence
differentiating with respect to $\pd/ \pd t_i$, we get 
\[ \pd p / \pd t_i \cdot h^r + p \cdot rx_ih^{r-1} \in V_{\kr(\tt)}, \]
and this gives that $p \cdot x_i$ is in $(V_{\kr(\tt)}:h^{r-1})_{|h}$. This shows the claim.

\medskip
Since we assume that $V$ is in general coordinates, we can assume that $h = x_n$ behaves
generically in the claim above. Thus
\begin{equation} 
I = \oplus_{r = 1}^d (V: x_n^r)_{|x_n} \oplus (I_V)_{|x_n} \sus \kr [x_1, \ldots, x_{n-1}] \label{MaLaIvx}
\end{equation}
is an ideal. Note that $\ini{I}$ is equal to $\gJ(\ini{V})_{x_n}$.

Now we claim that $\ini{I}$ is strongly stable. This is so in degrees $e \geq d$ because 
$\ini{I}_e = \ini{I_V}_{|x_n,e}$.
Since $V$ is in general coordinates, $\ini{I_V}$ is the generic initial ideal and so is strongly stable.
But then the restriction is also strongly stable.

Now $\ini{I}$ is also strongly stable in degrees $\leq d$. This follows by (\ref{MaLaIvx})
since $\ini{V}$ is strongly stable and so $\ini{V}:x_n^r$, and in the end the restriction is 
also strongly stable.

If therefore $\gJ(\ini{V})_{x_n}$ has no generators in degree $d$, we get by Greens Theorem \ref{MaTheGr}, 
that $\ini{I}$ is equal to $\gJ(\ini{V})_{x_n}$ and so
$\ini{(I_V)_{|x_n}}$ is equal to the latter in degrees $\geq d$.


\medskip
b. Let $W = V_{|x_n, \ldots, x_{n-r+1}}$. Since we are using the revlex order
\[ \ini{W} = \ini{V}_{|x_n, \ldots, x_{n-r+1}}. \]
Hence we are assuming that  that 
\[(\ini{W} : x_{n-r}) \cdot  \langle x_1, \ldots, x_{n-r} \rangle = \ini{W}.\]
But then by part a. we get 
\[ \gJ(\ini{W})_{x_{n-r}} = \ini{(I_W)_{|x_{n-r}}} \]
in degrees $\geq d$, and this gives b.
\end{proof}

\begin{corollary} 

a. If $\gJ(\ini{V})_{x_n}$ has codimension $c \leq n-2$, the vanishing locus of $V$ har codimension $c$
in ${\mathbb P}^{n-1}$. 

b. More generally, if $\gJ(\ini{V}_{|x_n, \ldots, x_{n-r+1}})_{x_{n-r}}$ has codimension $ c \leq n-r-2$, 
the vanishing locus of $V$ has codimension $c$ in ${\mathbb P}^{n-1}$. 
\end{corollary}
 
\begin{proof} 
a. Clearly $I_{V|x_n}$ and $\gJ(\ini{V})_{x_n}$ have the same codimension $c$. Since $c \leq n-2$,
the vanishing locus of $I_{V|x_n}$ will be nonempty in ${\mathbb P}^{n-2}$. 
Now since $V$ is in general coordinates, $x_n$ may be considered a general linear form and so
$c$ is also the codimension of the vanishing locus of $V$. 

b. Again $I_{V|x_n, \ldots, x_{n-r}}$ will have codimension $c$, and the vanishing locus will be nonempty in 
${\mathbb P}^{n-r-2}$. Since $V$ is in general coordinates, $x_n, \ldots, x_{n-r}$ may be considered
as general linear forms, and so $c$ is also the codimension of the vanishing locus of $V$. 

\end{proof}

Let us now show that Theorem \ref{MaTheFl} is also a corollary of the Main Theorem.

\begin{proof}[Proof of Theorem \ref{MaTheFl}]
Assume that $V$ is in general coordinates and let $m = n-r$. Then 
\[ \ini{V}_{|x_n, \ldots, x_{n-r+1}} : x_{n-r} = \langle x_1, \ldots, x_{n-r} \rangle ^{b-1} \cdot x_1^a. \]
Hence the conditions of part b. in the Main Theorem holds. Also since $\gJ(\ini{V}_{|x_n, \ldots, x_{n-r+1}})$
is the ideal generated by $x_1^a$, it will have codimension one. So if the codimension 
$1 \leq n-r-2 = 
m-2$, the vanishing locus of $V$ has codimension one and so there is a polynomial $p$ which is a
common factor of the forms in $V$. Clearly $\deg p \leq a$. If $\deg p = a^\prime < a$
we may write $V = p \cdot V^\prime$ where $V^\prime$ is a space of forms of degree $b+a-a^\prime$.
But then $\ini{V} = x_1^{a-a^\prime} \cdot \langle x_1, \ldots, x_m \rangle ^b$
and by applying the Main Theorem to $V^\prime$ we would again find a common factor. Continuing,
taking products of common factors, we would in the end have a polynomial of degree $a$ which is
a common factor.
\end{proof}



\subsection{Some examples.}
The Main Theorem is concerned with the situation that $\gJ(\ini{V})_{x_n}$ has no generators in 
degree $d$.
In the following examples we discuss situations where $\gJ(\ini{V})_{x_n}$ has 
one new generator in degree $d$ and a single generator in 
degrees less than $d$.

\eks Suppose $V$ is a subspace of $\kr[x_1, x_2, x_3]_3$ and $\gin{V}$ is given by the following
figure. 

\MaFigEkstra

\noindent So $\gin{V}$ is the stable subspace generated by $x_1^2x_3$ and $x_1x_2^2$.
Note that $\gJ(\ini{V})_{x_3}$ has one generator of degree $2$ and one generator of degree $3$.
We give three examples of $V$'s with this generic initial space.

\begin{theorem}
The following spaces $V$ will have $\gin{V}$ given by the figure above.
\begin{itemize}
\item[a.] $V = \langle x_1 q, x_2 q, x_3 q, p \rangle$ where $q$ is a quadratic form and $p$
is a cubic form not divided by $q$. Note that if $p$ and $q$ are relatively prime, the vanishing locus of $V$
is the complete intersection of these forms.
\item[b.] $V = \langle x_1(x_2^2+x_3^2), x_2(x_1^2 + x_3^2), x_3(x_1^2 + x_2^2), x_1x_2x_3 \rangle.$
The vanishing locus of $V$ consists of the three unit coordinate vectors.
\item[c.] $V = \langle x_1^3 + x_2^3 + x_3^3, x_1^2x_2 + x_2^2x_3 + x_3^2x_1, x_1x_2^2 + x_2x_3^2 + x_3x_1^2, x_1x_2x_3
\rangle .$ The vanishing locus of $V$ is empty.
\end{itemize}
\end{theorem}

In particular c. is striking in that the vanishing locus is empty, even if $\gJ(\ini{V})$ has only two generators, one in degree $d-1 = 2$ and one in
degree $d=3$.

\begin{proof}
Note that if $\gin{V}$ has the form above, then $g.V : x_3$ is one-dimensional for a general
change of coordinates $g$. Conversely, when $V$ is a four-dimensional space of forms of degree $3$ with 
$g.V : x_3$ one-dimensional for a general change of coordinates, then $\gin{V}$ will contain
$x_1^2x_3$. Since it is strongly stable, it must be of the form in the figure above. 

Now let $\kr(\tt)$ be the function field of $\kr [t_1,t_2, t_3]$, let
$V_{\kr(\tt)} = V \te_{\kr} \kr(\tt)$ and let $h = t_1x_1 + t_2x_2 + t_3x_3$. We must find
a four-dimensional $V$ such that $(V_{\kr(\tt)} : h)$ is one-dimensional. Let this latter space
be generated by a form $p$ in $\kr(\tt)$. In $ph$, the coefficients of the monomials in $t_1,t_2,t_3$ 
must span a subspace of $\kr[x_1,x_2,x_3]$ which is at most four-dimensional.

\medskip
a. Suppose the $\tt$-degree of $p$ is zero. Then $p\cdot h$ is $t_1 px_1 + t_2 px_2 + t_3 p x_3$.
The coefficients here span a three-dimensional space. Adding a form $q$ of degree $3$ not contained
in $p \cdot \langle x_1, x_2, x_3 \rangle$, we get a space $V = \langle x_1p, x_2p, x_3p, q \rangle$
whose generic initial space must be given by the figure above.

b. Suppose the $\tt$-degree of $p$ is one. Then we may write $p$ as 
$t_1 p_1 + t_2 p_2 + t_3 p_3$ where the $p_i$ are polynomials in the $x$'s. The coefficients
of the monomials in the $t$'s are
\begin{equation}\label{MainLigtto} 
x_1 p_1,\; x_1p_2 + x_2 p_1,\; x_1 p_3 + x_3p_1,\; x_2p_2,\; x_2p_3 + x_3 p_2,\; x_3p_3, 
\end{equation}
That these span at most a four-dimensional subspace, means that there are at least two linear
relations between them, or at least two linear syzygies of $p_1, p_2,$ and $p_3$. But such
$p$'s may be found as $2\times 2$-minors of a $2\times 3$ matrix whose rows are these two 
linear syzygies. For instance we may take the matrix
\[ \begin{pmatrix} x_1 & x_2 & 0 \\  0 & x_2 & x_3 \end{pmatrix}. \]
This gives 
$p_1 = x_2x_3, p_2 = x_1x_3,$ and $p_3 = x_1x_2$. It is then readily checked that the six polynomials
in (\ref{MainLigtto}) span a four-dimensional vector space which is the one given in b.

c. Suppose the $t$-degree of $p$ is two. After a somewhat intricate consideration, we found
that 
\[ p=(t_1x_2 + t_2x_3 + t_3 x_1)(t_1x_3+t_2x_1 + t_3x_2) \]
gives an example. It is easy to verify that the coefficients of the $\tt$-monomials in $ph$ are given by
four forms generating $V$, given in c. 
\end{proof}

Note that in b. and c. the examples are invariant
under the cyclic group $C_3$.

\medskip
Let $p$ be a form of degree $e$ and $q$ a sufficiently general form of degree $d > e$. 
If $V = ( p \cdot \langle x_1, x_2, x_3 \rangle^{d-e}, q)$, then $\gin{V}$ 
is the strongly stable set of monomials generated by $x_1^ex_3^{d-e}$ and $x_1^{e-1}x_2^{d-e+1}$, 
one generator of degree $e$ and one of degree $d$.
One might hope that conversely, if $e$ is sufficiently smaller than
$d$, and $\gin{V}$ is the above strongly stable set, this would imply that $V$ has a codimension 
one subspace $W$ such that $W = p \cdot \langle x_1, x_2, x_3 \rangle^{d-e}$ for some $p$. 
This is equivalent to there existing a subspace $W$ such that $\gin{W}$ is the stable subspace
generated by $x_1^ex_3^{d-e}$.

\eks Suppose $V$ is a subspace of $\kr[x_1, x_2, x_3]_4$ and  $\gin{V}$ is given by the following figure.

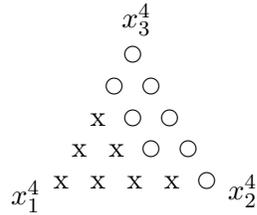
\MaFigGape

\noindent So $\gin{V}$ is the strongly stable subset generated by $x_1^2x_3^2$ and $x_1x_2^3$.
Note that $\gJ(\ini{V})_{x_3}$ has one generator of degree $2$ and one generator of degree $4$.
Could we then conclude that $V$ has a subspace $W$ of codimension one such that $\gin{W}$ is 
the stable subspace generated by $x_1^2x_3^2$ ? (This would imply that the polynomials in $W$
has a common factor.) We do not know the answer to this. 
But suppose $V$ is in general coordinates, so we may speak 
of initial ideals instead of generic initial ideals.
The initial space $\ini{V :x_3}$ is 
\begin{equation} \langle x_1^3, x_1^2x_2, x_1^2x_3 \rangle, \label{MaLaInx1}
\end{equation}
the stable subspace generated by $x_1^2 x_3$.
If we knew that
this was the generic initial space of $V:x_3$, we could by the techniques of Theorem \ref{MaTheMa}
have found such a $W$. The problem is however, that even if $V$ is in general coordinates, it 
is not certain that $V:x_3$ is, so it is not clear that $\gin{V:x_3}$ is (\ref{MaLaInx1}).
In fact, given that the dimension of $V$ is seven, the form of $\gin{V}$ given in the figure above, 
is equivalent to $V:h^2$ being nonzero for a general $h$.

That $\gin{V:x_3}$ is given by (\ref{MaLaInx1}) 
is equivalent to $\gin{V:h_1}$ being given by (\ref{MaLaInx1})  for a general linear form $h_1$.
But this is again equivalent to $V : h_1h_2$ being nonzero, where $h_1$ and $h_2$ are general linear
forms. It may not be likely that for a seven dimensional space $V$ of forms,
the condition of $V : h_1h_2$ being nonzero for general linear 
forms $h_1$ and $h_2$,
is equivalent to the condition of $V:h^ 2$ being nonzero for a general 
linear form $h$.

\eksfin

\end{document}